\documentclass[a4paper,12pt]{article}
\usepackage[centertags]{amsmath}
\usepackage{amsfonts}
\usepackage{amssymb}
\usepackage{amsthm}
\usepackage{dsfont}

\usepackage{wasysym}
\usepackage[draft]{optional}
\usepackage[usenames]{color}

\theoremstyle{plain}
\newtheorem{lemma}{Lemma}[section]
\newtheorem{theorem}[lemma]{Theorem}

\newtheorem{cor}[lemma]{Corollary}

\theoremstyle{definition}

\newtheorem{remark}[lemma]{Remark}

\newcommand{\R}{{\mathbb{R}}}

\begin{document}
\title{xxx}
\date{\today}
 \title{Decay estimates for Rivi\`ere's equation, with applications to regularity and compactness}
\author{Ben Sharp and Peter Topping}
\maketitle 

\begin{abstract}
We derive a selection of energy estimates for a generalisation
of a critical equation on the unit disc in $\mathbb{R}^2$ 
introduced by Rivi\`ere. Applications
include sharp regularity results and compactness theorems which
generalise a large amount of previous geometric PDE theory, 
including some of the theory of harmonic and almost-harmonic 
maps from surfaces.
\end{abstract}

\section{Introduction}

Suppose $u\in W^{1,2}(B_1,\mathbb{R}^m)$ is a weak solution to 
\begin{equation} \label{riv_eqn}
-\Delta u = \Omega.\nabla u
\end{equation}
where here and throughout this paper 
$B_1$ is the unit disc in $\R^2$, 
$\Omega \in L^2(B_1,so(m)\otimes \mathbb{R}^2)$, and we are
using the notation 
$[\Omega .\nabla u]^i=\langle \Omega^i_j,\nabla u^j \rangle$.
This equation, first considered in this generality by Rivi\`ere
\cite{riviere_inventiones}, generalises a number of interesting
equations appearing naturally in geometry, including the harmonic
map equation, the $H$-surface equation and, more generally, the
Euler-Lagrange equation of any conformally invariant
elliptic Lagrangian which is quadratic in the gradient.
A central issue is the regularity of $u$ implied by virtue of it
satisfying the equation \eqref{riv_eqn}. 
\emph{A priori}, the right-hand side of the equation looks like quite a general $L^1$ function, and 
standard elliptic regularity theory does not seem to help.
However, Rivi\`ere \cite{riviere_inventiones}
showed that any solution must necessarily be
continuous and even in 
$W^{2,1+\epsilon}$ for some $\epsilon>0$ \cite{riviere_notes}, thus generalising the famous regularity theory of
H\'elein \cite{helein_regularity}, for example. 
In most known interesting special cases of this equation,
one happens to know that $|\Omega|$ can be estimated linearly
in terms of $|\nabla u|$, i.e. we have 
$|\Omega.\nabla u|\leq C|\nabla u|^2$
and then a standard bootstrapping
argument can be applied to improve
the regularity of $u$ first to $C^{1,\alpha}$, and then
(via Schauder) to smoothness.

In this paper we investigate what sort of regularity and 
compactness properties we can deduce for solutions of the 
general equation \eqref{riv_eqn}, and even more general
inhomogeneous equations with the same special structure.
It is easy to convince oneself that it is unreasonable to
expect regularity better than $W^{2,2}$ in general.
%\cmt{In fact there is a 
%solution which is not in $W^{2,(2,\infty)}$ (the 
%Sobolev space of functions with two weak derivatives in $L^%{2,\infty}$).
%Moreover this same solution has derivatives that are not in $BMO$-see 
%Appendices \ref{Hardy}, \ref{Lorentz} for definitions if necessary.} 
However,
we will show that we \emph{do} have regularity up to this level,
or the best possible regularity when there is an inhomogeneity.

\begin{theorem}
\label{optimal_regularity}
Suppose that $u\in W^{1,2}(B_1,\mathbb{R}^m)$ 
is a weak solution on the unit disc in $\mathbb{R}^2$ to 
\begin{equation}\label{eqn2}
-\Delta u = \Omega.\nabla u +f, \,\,\,\,\,\,\,\,\,\,\, f \in L^p(B_1,\mathbb{R}^m)
\end{equation}
where $\Omega \in L^2(B_1,so(m)\otimes \mathbb{R}^2)$
and $p\in (1,2)$. Then 
$u\in W_{loc}^{2,p}(B_1)$. In particular, if $f\equiv 0$, then
$u\in W_{loc}^{2,p}$ for all $p\in [1,2)$ and 
$u\in W_{loc}^{1,q}$ for all $q\in [1,\infty)$.

Moreover, for $U\subset\subset B_1$, there exist 
$\eta_0=\eta_0(p,m)>0$ and $C=C(p,m,U)<\infty$
so that if $\|\Omega\|_{L^2(B_1)}\leq \eta_0$ then
\begin{equation}
\label{w2pest}
\|u\|_{W^{2, p}(U)} 
\leq C(\| f\|_{L^p(B_1)} + \|u\|_{L^1(B_1)}).
\end{equation}
\end{theorem}
This theorem omits the borderline case $p=2$ for good reason;
even in the case that $f\equiv 0$, one can find solutions
so that $u$ is neither $W^{2,2}$ nor Lipschitz. Moreover,
examples with $f\equiv 0$ show that the first derivatives 
of $u$ need not even lie in $BMO$, and (consequently) the second 
derivatives need not even lie in the Lorentz space 
$L^{2,\infty}$ (see Appendices \ref{Hardy} and \ref{Lorentz} 
for definitions if necessary).

As a corollary of our theorem, we see that $f\in L^p$ implies that
$u$ lies in $C^{0,2(1-\frac{1}{p})}$, hence recovering a
result of Rupflin \cite{rupflin_calcvar}
in the case of two-dimensional domains. 
Rivi\`ere has informed us that our regularity assertion in
the particular case $f\equiv 0$ will also be made in 
the final version of \cite{riviere_notes}, based on a different
proof.

We remark that the estimate \eqref{w2pest} fails without
the smallness of $\Omega$ hypothesis.
More precisely, there exist a sequence 
$\Omega_k\in L^2(B_1,so(m)\otimes \mathbb{R}^2)$ uniformly bounded
in $L^2$, and a sequence of weak solutions 
$u_k\in W^{1,2}(B_1,\mathbb{R}^m)$ to the equation
$$-\Delta u_k = \Omega_k.\nabla u_k,$$
uniformly bounded in $W^{1,2}$, such that $u_k$ is unbounded
in any $W^{2,p}$ space with $p\in (1,2)$.
(A sequence of harmonic maps undergoing bubbling would
provide an example.)

Estimate \eqref{w2pest} implies that for any 
sequence $\Omega_k\in L^2(B_1,so(m)\otimes \mathbb{R}^2)$
with $\|\Omega_k\|_{L^2(B_1)}\leq\eta_0$, and any sequence
of weak solutions $u_k\in W^{1,2}(B_1,\mathbb{R}^m)$ 
to the equation
$$-\Delta u_k = \Omega_k.\nabla u_k+f_k,$$
with $u_k$ uniformly bounded in $W^{1,2}$ and $f_k$ uniformly bounded in some space $L^p$ for $p\in (1,2)$, we may deduce that
$u_k$ is locally uniformly bounded in $W^{2,p}$.
By the theorem of Rellich-Kondrachov, we can deduce that $u_k$
is precompact in $W^{1,t}(B_{1/2})$ for any $t < \frac{2p}{2-p}$.

In this paper, we work somewhat harder to prove a stronger
compactness result, extending a recent theorem of 
Li and Zhu \cite{LZ}, in which we assume merely that the inhomogeneous
terms $f_k$ are bounded in $L\ln L$ (a space larger than any of the $L^p$ spaces with $p>1$, but slightly smaller than $L^1$;
see Appendix \ref{Lorentz} for more information on this space, 
and definitions, if necessary).

\begin{theorem}[Compactness]\label{imp}
Suppose that 
we have a sequence $\lbrace u_n \rbrace  \subset W^{1,2}(B_1,\mathbb{R}^m)$ of weak solutions to 
\begin{equation*}
 -\Delta u_n = \Omega _n .\nabla u_n + f_n
\end{equation*}
on the unit disc in $\mathbb{R}^2$, 
where $\{\Omega_n\} \subset  L^2(B_1,so(m)\otimes \mathbb{R}^2)$ and $\{ f_n \} \subset L\ln L(B_1,\mathbb{R}^m)$. 
Suppose also that there exists $\Lambda < \infty$ such that 
\begin{equation*}
 \|u_n \|_{L^1(B_1)} + \|f_n\|_{L\ln L(B_1)} \leq \Lambda.
\end{equation*}
Then there exist an $\eta_2 = \eta_2 (m)>0 $ and $u\in W_{loc}^{1,2}(B_1,\mathbb{R}^m)$ such that if $\|\Omega_n\|_{L^2(B_1)} \leq \eta_2$ then after passing to a subsequence
\begin{equation*}
 \lim_{n\rightarrow \infty} \|u_n - u \|_{W^{1,2}(B_{1/2})}=0.
\end{equation*}
\end{theorem}

We will show in Section \ref{H1doesntwork} 
that this result fails if we replace
$L\ln L$ by the related Hardy space $h^1$ 
(see Appendix \ref{Hardy}).
In the special case that $\{\Omega_n\}$ is a precompact set in
$L^1$, and $u_n$ is uniformly bounded in $W^{1,2}$, 
this result was proved recently by Li and Zhu \cite{LZ}.

\begin{remark}
\label{W21rmk}
The compactness result is ruling out concentration of energy as is done in \cite{LZ} - i.e. concentration of $\|\nabla u_n\|_{L^2}^2$. In contrast, we do not rule out concentration of $\|\nabla u_n\|_{L^{2,1}}^2$ or of the corresponding second order quantity
$\|\nabla^2 u_n\|_{L^1}$. However, it will follow from our estimates
(and in particular, \eqref{int2} below) that if these latter
concentrations occur we must have $f_n$ concentrating
in $L\ln L$. 
\end{remark}

Even in the classical case that $\Omega\equiv 0$
there is a consequence of such compactness which may be 
worth remarking,
although one which would follow from previously known theory.

\begin{cor}
\label{compact_embedding}
On the ball in $\mathbb{R}^2$, the embedding 
$$L\ln L(B_1)\hookrightarrow H^{-1}(B_1)$$
is compact.
\end{cor}

At the heart of this paper is a collection of energy/decay
estimates which we summarise in the following theorem.

\begin{theorem}[Main supporting theorem]\label{mesest}
 Suppose $u\in W^{1,2}(B_1,\mathbb{R}^m)$ is a weak solution to 
\begin{equation}\label{eqn3}
-\Delta u = \Omega.\nabla u + f, \,\,\,\,\,\,\,\,\,\,\, f \in L \ln L(B_1,\mathbb{R}^m) 
\end{equation}
on the unit disc in $\mathbb{R}^2$, where $\Omega \in L^2(B_1,so(m)\otimes \mathbb{R}^2)$.
Writing $\bar u = \frac{1}{|B_1|}\int_{B_1}u$,
\begin{enumerate}
\item
there exist $\eta = \eta (m) > 0$ and $K_1=K_1(m)<\infty$
such that if $\|\Omega\|_{L^2(B_1)} \leq \eta$, then for all $r\in (0,1/2]$ we have
\begin{equation}
\label{int1}
 \|\nabla u \|_{L^2(B_r)}^2 \leq  K_1\left( \|\Omega\|_{L^2(B_1)}^2 \|\nabla u\|_{L^2(B_1)}^2 + r^2 \|u-\bar u\|_{L^1(B_1)}^2 + \|f\|_{L^1(B_1)} \|f\|_{L\ln L(B_1)}\right)
\end{equation}
and 
\begin{equation}
\label{int2}
 \|\nabla^2 u \|_{L^1(B_r)} \leq K_1\left(\|\Omega \|_{L^2(B_1)} \|\nabla u \|_{L^2(B_1)} + r^2\|u- \bar u\|_{L^1(B_1)} + \|f\|_{L\ln L(B_1)}\right);
\end{equation}
\item
for all $\delta>0$ there 
exist $\eta = \eta (m,\delta) > 0$ and $K_2=K_2(m,\delta)<\infty$
such that if $\|\Omega\|_{L^2(B_1)} \leq \eta$, then for all 
$r\in (0,1]$ we have
\begin{eqnarray}
\label{est2}
\|\nabla u \|_{L^2(B_r)}^2 &\leq&  
(1+\delta)r^2 \|\nabla u\|_{L^2(B_1)}^2 + \nonumber \\ 
&+& K_2\left( \|\Omega\|_{L^2(B_1)}^2 \|\nabla u\|_{L^2(B_1)}^2 +  \|f\|_{L^1(B_1)} \|f\|_{L\ln L(B_1)}\right).
\end{eqnarray}
\end{enumerate}
\end{theorem}

Although we will not need it in this work, we note that 
the first part of the theorem will also yield estimates for
$\nabla u$ in the Lorentz space $L^{2,1}$ (by the embedding $W^{1,1}\hookrightarrow L^{2,1}$). 

The estimates of the first part of the theorem are
interior estimates which have the weakest norms of $u$ on the 
right-hand side.
By combining them with a standard covering argument, we will
also derive the following optimal global estimate:

\begin{theorem}\label{betest}
With $u$ and $f$ as in Theorem \ref{mesest} and 
$U\subset \subset B_1$ there exist an $\eta_1 =\eta_1(m)>0$ and $C=C(m, U) <\infty$  such that if $\|\Omega\|_{L^2(B_1)} \leq \eta_1$, then
\begin{equation*}
 \|u\|_{W^{2,1}(U)} \leq C \left( \|u\|_{L^1(B_1)} + \|f\|_{L\ln L(B_1)} \right).
\end{equation*}
 
\end{theorem}
 
The second part of Theorem \ref{mesest} is used to obtain 
both the regularity result Theorem \ref{optimal_regularity} and 
the compactness result Theorem \ref{imp}.

We remark that Theorems \ref{optimal_regularity}, \ref{imp}, \ref{mesest} and \ref{betest} all fail if 
we drop the antisymmetry hypothesis on $\Omega$. 

The paper is laid out as follows.
In Section \ref{proof_mesest} we prove the main supporting Theorem \ref{mesest}, which is central to the other results, then we go on to prove Theorem \ref{betest} in Section \ref{proof_betest}. This allows us to prove the compactness Theorem \ref{imp} in Section \ref{proof_imp} and its corollary in Section \ref{proof_corollary}. We leave the regularity Theorem \ref{optimal_regularity} to Section \ref{proof_optimal_regularity} and finally in Section \ref{H1doesntwork} we give an example to show that the compactness result fails if we replace $L\ln L$ by the Hardy space $h^1$. 

\emph{Acknowledgments:} Both authors were supported by 
The Leverhulme Trust. The second author would like to thank
Pawe\l{} Strzelecki for useful discussions.

\section{Preliminaries}

To begin, we describe some properties of the space $L\ln L$,
and record the behaviour of the equations \eqref{eqn2}, \eqref{eqn3} and 
various norms under scaling.
In addition, we have collected a number of known results 
in an appendix.

\subsection{Estimates for $L\ln L$}
For the definition of $L\ln L$ and $f^*$, 
see Appendix \ref{Lorentz}.

\begin{lemma}\label{l1-llogl}
 Suppose $f \in L \ln L(B_r(x_0))$ and $r \in (0,1/2]$. Then there exists $C< \infty$ such that 
\begin{equation*}
 \|f\|_{L^1(B_r(x_0))} \leq C \left[\ln \left(\frac{1}{r} \right) \right]^{-1} \|f\|_{L \ln L(B_r(x_0))}.
\end{equation*}

\end{lemma}

\begin{proof}
Notice that 
\begin{eqnarray*}
0 &\leq& r^2 \int_0^{|B_1|} f^*(r^2 t) \ln \left(2 + \frac{1}{t}\right) \,dt \\
&=&  \int_0^{|B_r(x_0)|} f^*(s) \ln \left(2 + \frac{r^2}{s}\right) \,ds \\
&=&  \int_0^{|B_r(x_0)|} f^*(s) \ln r^2 \, ds +   \int_0^{|B_r(x_0)|} f^*(s) \ln \left(\frac{2}{r^2} + \frac{1}{s}\right) \, ds \\
&\leq& -2\ln \left(\frac{1}{r}\right) \|f\|_{L^1(B_r(x_0))} + C\|f\|_{L \ln L(B_r(x_0))},
\end{eqnarray*}
where the final inequality is obtained by noticing that $s\leq \pi r^2 <1$ which implies $\frac{2}{r^2} + \frac{1}{s} \leq \frac{2\pi +1}{s} \leq \left(2 + \frac{1}{s} \right)^C$ for some fixed $C$.  
\end{proof}

The following lemma indicates that $L\ln L$ norms do not 
deteriorate under scaling. However we emphasise that they 
need not improve, unlike $L^p$ norms for $p>1$.

\begin{lemma}\label{fhat}
  Suppose $f \in L \ln L(B_r(x_0))$ where $r \in (0,1/2]$. 
  Defining $\hat{f}:= r^2 f(x_0 + rx)$ there exists $C < \infty$ such that
\begin{equation*}
 \|\hat{f}\|_{L \ln L(B_1)} \leq C\|f\|_{L \ln L(B_r(x_0))}.
\end{equation*}
\end{lemma}

\begin{proof}
First we calculate
\begin{eqnarray*}
 \hat{f}^*(t) &=& \inf \{ s \geq 0 : |\{ x \in B_1 : |\hat{f} (x) |>s\}|\leq t \} \\
 &=& \inf \{ s \geq 0 : |\{ x \in B_1 : |r^2 f(x_0 + rx) |>s\}|\leq t \} \\
&=& \inf \{ s \geq 0 : |\{ y \in B_r(x_0) : | f(y) |>\frac{s}{r^2}\}|r^{-2} \leq t \} \\
&=& r^2 f^* (r^2 t)
\end{eqnarray*} 
therefore
\begin{eqnarray*}
 \|\hat{f}\|_{L \ln L(B_1)} &\leq& C\int_0^{|B_1|} \hat{f}^*(t) \ln \left(2 + \frac{1}{t}\right) \, dt \\
&=& C\int_0^{|B_1|} r^2 f^*(r^2 t) \ln \left(2 + \frac{1}{t}\right) \, dt \\
&=& C\int_0^{|B_r(x_0)|} f^*(s) \ln \left(2 + \frac{r^2}{s}\right) \, ds \\
&\leq& C\|f\|_{L \ln L(B_r(x_0))}.
\end{eqnarray*}
\end{proof}

\subsection{Scaling}
\label{scalingsect}

There will be several occasions when we will require estimates
on some small ball $B_{R/2}(x_0)$ in terms of quantities on
the ball $B_R(x_0)$.
In this section, we make a note of what scaling we will be taking,
and how each relevant quantity, and the equation itself, behave
under this operation.

Let $u$ be a solution to \eqref{eqn2} or \eqref{eqn3}. 
For $x_0 \in B_1$, let $R>0$ be such that 
$B_{R}(x_0) \subset B_1$. 
Now we rescale $u$ by defining $\hat{u}(x):= u(x_0 + Rx) $, 
$\hat{f} (x) := R^2 f(x_0 + Rx)$ and 
$\hat{\Omega} (x) := R \Omega (x_0 + Rx)$. 

We have $\hat{u} \in W^{1,2}(B_1, \mathbb{R}^m)$ and 
\begin{eqnarray*}
 -\Delta \hat{u} (x) &=& -R^2 \Delta u (x_0 + Rx) \\ 
&=& R\Omega (x_0 + Rx).R\nabla u (x_0 + Rx)  + R^{2} f(x_0 + Rx) \\ 
&=& \hat{\Omega} (x).\nabla \hat{u} (x) + \hat{f} (x),
\end{eqnarray*}
i.e. the same equation as before.
The quantities of which we will need to keep track are:
\begin{enumerate}
\item
$\|\nabla\hat{u}\|_{L^p(B_r)} =R^{1-\frac{2}{p}} \|\nabla u\|_{L^p(B_{rR}(x_0))}$ 
for any $r\in [0,1]$
\item
$\| \hat{\Omega} \|_{L^2(B_1)}=\| \Omega \|_{L^2(B_R(x_0))}$
% \leq \eta_1$, 
\item
$\|\hat{u}\|_{L^p(B_1)} = R^{-\frac{2}{p}} \|u\|_{L^p(B_R(x_0))}$ 
\item
$\|\hat{f}\|_{L^p(B_1)} =  
R^{2(1-\frac{1}{p})}\|f\|_{L^p(B_R(x_0))}$
\item
$\|\hat{f}\|_{L \ln L (B_1)} \leq  C\|f\|_{L \ln L (B_R(x_0))}$ 
\end{enumerate}
where the final estimate is following from Lemma \ref{fhat}.

\section{Proof of the decay estimates, Theorem \ref{mesest}}
\label{proof_mesest}
Most of the work in the proof will be common to both parts of
the theorem. We will be referring to the $\delta$ of the second
part with the understanding that in the case of the first part, 
we could just set $\delta=1$.

We start off with $\eta=\epsilon$, taken from 
Lemma \ref{Riv}, and will assume throughout that
$\|\Omega\|_{L^2(B_1)} \leq \eta$, with the understanding
that the upper bound $\eta$ will be lowered at different points during the proof.
For our weak solution $u$ to \eqref{eqn3} corresponding 
to $\Omega$, we will assume, without loss of generality, 
that $\bar u = \frac{1}{|B_1|}\int_{B_1} u =0$.  

To begin with we use Rivi\`{e}re's decomposition of $\Omega$ (Lemma \ref{Riv}) in order to rewrite the equation \eqref{eqn3}
(equations \eqref{re1} and \eqref{re2} below).
Lemma \ref{Riv} gives us 
$A\in W^{1,2}(B_1, GL_m (\mathbb{R}))\cap L^{\infty} (B_1, GL_m (\mathbb{R}))$, $B\in W^{1,2}(B_1, gl_m (\mathbb{R}))$ and $C=C(m)<\infty$ so that
\begin{equation*}
 \nabla A - A\Omega = \nabla^{\bot} B
\end{equation*}
and 
\begin{equation*}
 \|\nabla A\|_{L^2(B_1)} + \|\nabla B\|_{L^2(B_1)} + \|dist(A,SO(m)\|_{L^{\infty}(B_1)} \leq C \|\Omega\|_{L^2(B_1)}.
\end{equation*}

Now, 
\begin{eqnarray}\label{re1}
 \text{div} (A \nabla u) &=& \nabla A. \nabla  u + A \Delta u  \nonumber \\
 &=& \nabla A. \nabla u  - A \Omega.\nabla u - Af  \nonumber \\ 
&=& \nabla^{\bot} B. \nabla u - Af
\end{eqnarray}

and
\begin{equation} \label{re2}
 \text{curl}(A \nabla u) = \nabla^{\bot}A.\nabla u.
\end{equation}
We note here that the above equations only hold in a weak sense, and more care should be taken in their calculation. We illustrate this for (\ref{re1}): \emph{A priori} $\text{div}(A\nabla u)$ is a distribution, so for $\phi \in C_c^{\infty}(B_1)$ we have 
\begin{eqnarray*}
 \text{div}(A\nabla u) [\phi] &=& -\int_{B_1} A\nabla u . \nabla \phi \\
&=& \int_{B_1} (\nabla A.\nabla u) \phi - \nabla (\phi A) . \nabla u \\
&=& \int_{B_1}  (\nabla A.\nabla u) \phi - (A\Omega .\nabla u) \phi - Af \phi \qquad\text{since}\,\, u \,\, \text{weakly solves (\ref{eqn3})} \\
&=& \int_{B_1} (\nabla^{\bot} B. \nabla u - Af) \phi = (\nabla^{\bot} B. \nabla u - Af) [\phi].
\end{eqnarray*}
 
We will now essentially carry out a Hodge decomposition of $A\nabla u$ in $B_1$ using the expressions (\ref{re1}) and (\ref{re2}). We first extend all the quantities arising above to functions on $\mathbb{R}^2$.

Let $Ex:W^{1,2}(B_1) \rightarrow W_{0}^{1,2}(\mathbb{R}^2)$ be a bounded extension operator
with each function in the image supported in $B_2$. 
Denote $\tilde{u} = Ex(u) \in W^{1,2}(\mathbb{R}^2, \mathbb{R}^m)$ and note that since we are assuming $\int_{B_1} u = 0$, by the Poincar\'{e} inequality and by standard properties of $Ex$ we have 
\begin{equation*}
 \|\tilde{u}\|_{W^{1,2}(\mathbb{R}^2)} \leq C\|u\|_{W^{1,2}(B_1)} \leq C \|\nabla u \|_{L^2(B_1)}
\end{equation*}
and $u = \tilde{u}$ in $B_1$.

For $A$, first let $\hat{A} = A - \frac{1}{|B_1|}\int_{B_1} A$ and $\tilde{A} = Ex (\hat{A}) \in W^{1,2}(\mathbb{R}^2, gl_m (\mathbb{R})) $. Noting that $\int_{B_1} \hat{A} =0$ and using the same argument as for $u$ we have 
\begin{equation*}
 \|\tilde{A}\|_{W^{1,2}(\mathbb{R}^2)} \leq C \|\nabla A\|_{L^2(B_1)}
\end{equation*}
here we have used that $\nabla \tilde{A} = \nabla \hat{A} = \nabla A$  in $B_1$. Notice also that $\tilde{A} \nabla \tilde{u} + \left( \frac{1}{|B_1|} \int_{B_1} A \right) \nabla \tilde{u} = A\nabla u$ in $B_1$. 

We carry out the same extension for $B$ to get $\tilde{B}$ as above for $A$. 
We extend $f$ by zero (without relabelling), so by 
Appendix \ref{Lorentz}, $f \in h^1(\R^2)$ with $\|f\|_{h^1(\R^2)} \leq C \|f\|_{L\ln L(B_1)}$. 
\newline

Now we define 
\begin{equation*}
D:= N [ \nabla^{\bot} \tilde{B} . \nabla \tilde{u}], 
\end{equation*}
\begin{equation*}
E:= N [ \nabla^{\bot} \tilde{A} . \nabla \tilde{u}],
\end{equation*}
\begin{equation*}
F:= -N[Af],
\end{equation*}
where $N$ is the Newtonian potential (see Appendix \ref{sing}).
Note that the quantity $Af$ is well defined on the whole of $\mathbb{R}^2$ by the definition of $f$.
Finally let 
\begin{equation*}
 H:= \tilde{A} \nabla \tilde{u} + \left( \frac{1}{|B_1|} \int_{B_1} A \right) \nabla \tilde{u} - \nabla D -\nabla F - \nabla^{\bot} E. 
\end{equation*}

The first thing to notice about $H$ is that
\begin{equation}\label{hodge}
 H = A \nabla u - \nabla D - \nabla F - \nabla^{\bot} E
\end{equation}
in $B_1$.
Hence we have
\begin{eqnarray*}
\text{div} (H) = \text{div} (A\nabla u) -\Delta (D+F) = \text{div} (A\nabla u) - \nabla^{\bot} \tilde{B}.\nabla \tilde{u} + A f = 0
\end{eqnarray*}
weakly in $B_1$, and a similar calculation shows  $\text{curl} (H) = 0$ weakly in $B_1$. (Again care must be taken in checking 
these.) Therefore $H$ is harmonic in $B_1$ (i.e. corresponds 
to a harmonic 1-form). 

Suppose $r\in (0,1]$. (For some estimates later it will need
to be less than $\frac{1}{2}$.)

Without loss of generality, we may assume that $\delta\in (0,1]$.
(Recall that when addressing the first part of the theorem, 
we are just setting $\delta=1$.)
For $\eta$ small enough, depending on $\delta$, we may assume (by the estimate in Lemma \ref{Riv}) that $A$ is close to a special-orthogonal matrix in the sense that both $A$ and $A^{-1}$ 
change the length of any vector by at most a factor of $1+\delta$. Therefore
\begin{equation}\label{GFH}
\begin{aligned}
 \| \nabla u \|_{L^2(B_r)}^2 &\leq (1+\delta)^2\|A\nabla u \|_{L^2(B_r)}^2 \\
&\leq  (1+3\delta)\|A\nabla u \|_{L^2(B_r)}^2\\ 
&\leq
(1+4\delta)\|H \|_{L^2(B_r)}^2+
C( \|\nabla D \|_{L^2(B_r)}^2 + \| \nabla F \|_{L^2(B_r)}^2 + \| \nabla E \|_{L^2(B_r)}^2),
\end{aligned}
\end{equation}
where $C$ is dependent on $\delta$.
In order to obtain the inequalities of Theorem \ref{mesest} we estimate $\|H \|_{L^2(B_r)}$, $\|\nabla D \|_{L^2(B_r)}$, $\| \nabla F \|_{L^2(B_r)}$ and $\| \nabla E \|_{L^2(B_r)}$. 

First we consider $\nabla D = \nabla N[ \nabla^{\bot} \tilde{B} . \nabla \tilde{u}] $ and $\nabla E = \nabla N [ \nabla^{\bot} \tilde{A} . \nabla \tilde{u}] $. Notice that by the work of
Coifman-Lions-Meyer-Semmes \cite{clms}  and the fact that $\nabla N : \mathcal{H}^1(\mathbb{R}^2) \rightarrow L^{2,1}(B_1)$ is a bounded linear operator (see Appendix \ref{misc}) we have, 
\begin{eqnarray}\label{D and E}
 \|\nabla D \|_{L^2(B_1)} + \| \nabla E \|_{L^2(B_1)} &\leq& C \left(\|\nabla D \|_{L^{2,1}(B_1)} + \| \nabla E \|_{L^{2,1}(B_1)}\right) \nonumber \\
&=& C \left(\|\nabla N[ \nabla^{\bot} \tilde{B} . \nabla \tilde{u}] \|_{L^{2,1}(B_1)} + \|\nabla N [ \nabla^{\bot} \tilde{A} . \nabla \tilde{u}] \|_{L^{2,1}(B_1)}\right) \nonumber \\
&\leq& C \left(\|\nabla^{\bot} \tilde{B} . \nabla \tilde{u}\|_{\mathcal{H}^1(\mathbb{R}^2)} + \|\nabla^{\bot} \tilde{A} . \nabla \tilde{u}\|_{\mathcal{H}^1(\mathbb{R}^2)}\right) \nonumber \\
&\leq& C \left( \|\nabla B \|_{L^2(B_1)} + \|\nabla A\|_{L^2(B_1)}\right)\|\nabla u \|_{L^2(B_1)} \nonumber \\
&\leq& C\|\Omega\|_{L^2(B_1)} \|\nabla u \|_{L^2(B_1)}
\end{eqnarray}
where we have also used the continuous embedding $L^{2,1} \hookrightarrow L^2$ and the estimate from Lemma \ref{Riv}. 

For $\nabla F = -\nabla N[Af]$ we use (Appendix \ref{misc})
that the Riesz potential $\nabla N : L^1(B_1) \rightarrow L^{2,\infty}(B_1)$ is a bounded operator; also $\nabla N : h^1(\R^2) \rightarrow L^{2,1}(B_1)$ is bounded. We will also use the following: $L^{2,\infty}$ is the dual of $L^{2,1}$; if $f \in L\ln L (B_1)$ then for any $g \in L^{\infty}$, $gf \in L\ln L (B_1)$ and $\|gf\|_{L\ln L(B_1)} \leq \|g\|_{L^{\infty}} \|f\|_{L\ln L(B_1)}$ and finally we use the continuous embedding $L\ln L(B_1) \hookrightarrow h^1(\R^2)$ (see 
Appendix \ref{Lorentz}). We have
\begin{eqnarray}\label{F}
 \|\nabla F\|_{L^2(B_1)}^2 &\leq& C\|\nabla F\|_{L^{2,\infty}(B_1)} \|\nabla F\|_{L^{2,1}(B_1)} \nonumber \\
 &\leq& C \|  A f \|_{L^1(B_1)} \|  A f \|_{h^1(\R^2)} \nonumber \\ 
&\leq& C \|f\|_{L^1(B_1)} \|  A f \|_{L\ln L (B_1)}  \nonumber \\
&\leq& C \|f\|_{L^1(B_1)} \|f\|_{L \ln L(B_1)}.
\end{eqnarray}
Also, using merely the boundedness of 
$\nabla N : L^1(B_1) \rightarrow L^{2,\infty}(B_1)$, we have
\begin{equation}
\label{FL1}
\|\nabla F\|_{L^1(B_1)} \leq 
C\|f\|_{L^1(B_1)}.
\end{equation}

From here, we proceed differently in order to prove the two different parts of the theorem.
For the first part, we now estimate $\|H\|_{L^1(B_{2/3})}$ and apply standard estimates for harmonic functions in order to estimate $\|H\|_{L^2(B_r)}$: Using Lemma \ref{nablauL}, and estimates 
\eqref{D and E} and \eqref{FL1}, we have
\begin{eqnarray}\label{0H}
 \|H\|_{L^1(B_{2/3})} &\leq& C \left( \|\nabla u\|_{L^1(B_{2/3})} + \|\nabla D\|_{L^1(B_{2/3})} + \|\nabla E\|_{L^1(B_{2/3})} + \|\nabla F \|_{L^1(B_{2/3})} \right)  \nonumber \\
&\leq& C \left( \|u\|_{L^1(B_1)} + \|f\|_{L^1(B_1)}+ \|\Omega\|_{L^2(B_1)} \|\nabla u \|_{L^2(B_1)}  \right).
\end{eqnarray}
Since $H$ is harmonic we have pointwise estimates on $H$ and its derivatives on the interior of $B_{2/3}$ in terms of $\|H\|_{L^1(B_{2/3})}$, and in particular 
\begin{equation}\label{nablaH}
 \|H\|_{L^{\infty}(B_{1/2})} + \|\nabla H\|_{L^{\infty}(B_{1/2})} \leq C \|H\|_{L^1(B_{2/3})}.
\end{equation}
Therefore if we consider $r\in (0,\frac 12]$, then
\begin{eqnarray}\label{H}
 \|H\|_{L^2(B_r)}^2 &\leq& \pi r^2 \|H\|_{L^{\infty}(B_r)}^2 \nonumber \\
 &\leq& Cr^2 \left( \|u\|_{L^1(B_1)}^2 + \|f\|_{L^1(B_1)}^2+ \|\Omega\|_{L^2(B_1)}^2 \|\nabla u \|_{L^2(B_1)}^2  \right).
\end{eqnarray}

Now, looking back at inequality (\ref{GFH}) and using (\ref{D and E}), (\ref{F}) and (\ref{H}) we have 
\begin{eqnarray*}
 \| \nabla u \|_{L^2(B_r)}^2 &\leq& C \left(\|\Omega\|_{L^2(B_1)}^2 \|\nabla u \|_{L^2(B_1)}^2 + r^2 \|u\|_{L^1(B_1)}^2 + \|f\|_{L^1(B_1)} \|f\|_{L\ln L(B_1)} \right)
\end{eqnarray*}
which is the first inequality that we seek from the first part 
of the theorem.

In order to get the second estimate \eqref{int2}
of the first part of the theorem,
we return to the Hodge decomposition \eqref{hodge} 
which tells us that
\begin{equation*}
 \nabla u = A^{-1}( H + \nabla D + \nabla F + \nabla^{\bot} E)
\end{equation*}
in $B_1$. 
Using the fact that the operators $\nabla^2 N : h^1(\R^2)\rightarrow L^1(B_1)$ and $\nabla^2 N: \mathcal{H}^1(\R^2) \rightarrow L^1(\R^2)$ are bounded (see Appendix \ref{Hardy}),
Lemma \ref{Riv} and equations \eqref{D and E}, \eqref{F}, \eqref{0H}, \eqref{nablaH} and \eqref{H} we find that
\begin{equation*}
 \nabla^2 u = \nabla A^{-1}.( H + \nabla D + \nabla F + \nabla^{\bot} E) + A^{-1}(\nabla H + \nabla^2 D + \nabla^2 F + \nabla\nabla^{\bot} E)
\end{equation*}
and 
\begin{equation*}
 \|\nabla^2 u \|_{L^1(B_r)} \leq C(\|\Omega \|_{L^2(B_1)} \|\nabla u \|_{L^2(B_1)} + r^2\|u\|_{L^1(B_1)} + r\|\Omega \|_{L^2(B_1)}\|u\|_{L^1(B_1)} + \|f\|_{L\ln L(B_1)}).
\end{equation*}
Since we have assumed without loss of generality 
that $\int_{B_1} u=0$, by an application of the Poincar\'{e} inequality we have
\begin{equation*}
 \|\nabla^2 u \|_{L^1(B_r)} \leq C(\|\Omega \|_{L^2(B_1)} \|\nabla u \|_{L^2(B_1)} + r^2\|u\|_{L^1(B_1)} + \|f\|_{L\ln L(B_1)}),
\end{equation*}
as desired.

For the second part of the theorem, 
we return to a general $r\in (0,1]$. We
will now control $H$ using the standard decay estimate
$$\|H\|_{L^2(B_r)}^2 \leq r^2 \|H\|_{L^2(B_1)}^2$$
which holds since $H$ is harmonic (see \cite[Lemma 3.3.12]{helein_conservation}).

Then using \eqref{GFH} and \eqref{D and E} and \eqref{F} again,
we find that
\begin{equation}
\begin{aligned}
 \| \nabla u \|_{L^2(B_r)}^2 &\leq 
(1+4\delta)\|H \|_{L^2(B_r)}^2+
C( \|\nabla D \|_{L^2(B_r)}^2 + \| \nabla F \|_{L^2(B_r)}^2 + \| \nabla E \|_{L^2(B_r)}^2)\\
&\leq
(1+4\delta)r^2\|H \|_{L^2(B_1)}^2+
C( \|\nabla D \|_{L^2(B_1)}^2 + \| \nabla F \|_{L^2(B_1)}^2 + \| \nabla E \|_{L^2(B_1)}^2)\\
&\leq
(1+5\delta)r^2\|A\nabla u \|_{L^2(B_1)}^2+
C( \|\nabla D \|_{L^2(B_1)}^2 + \| \nabla F \|_{L^2(B_1)}^2 + \| \nabla E \|_{L^2(B_1)}^2)\\
&\leq
(1+5\delta)(1+\delta)^2 r^2\|\nabla u \|_{L^2(B_1)}^2+
C( \|\nabla D \|_{L^2(B_1)}^2 + \| \nabla F \|_{L^2(B_1)}^2 + \| \nabla E \|_{L^2(B_1)}^2)\\
&\leq
(1+100\delta)r^2\|\nabla u \|_{L^2(B_1)}^2+
C\left(\|\Omega\|_{L^2(B_1)}^2 \|\nabla u \|_{L^2(B_1)}^2 
+ \|f\|_{L^1(B_1)} \|f\|_{L \ln L(B_1)} \right)
\end{aligned}\nonumber
\end{equation}
Thus, by repeating the argument with $\delta$ reduced by 
a factor of 100, we conclude the proof.

\section{Proof of the $W^{2,1}$ estimate, Theorem \ref{betest}}
\label{proof_betest}
We can say immediately that the $\eta_1$ whose existence
is claimed in the theorem can be chosen as
$\eta_1^2 = \min \left\{ \eta^2 ,  \frac{\epsilon_0}{K_1} \right\}$, 
where $\epsilon_0$ is that given in Lemma \ref{LS} corresponding to $k=4$, and 
$K_1$ and $\eta$ are from the first part of 
Theorem \ref{mesest}.

We would like to rescale the first estimate \eqref{int1} of the 
first part of Theorem \ref{mesest}, in the case that $r=\frac 12$.
Indeed, adopting the notation of Section \ref{scalingsect},
we know that
\begin{equation}
\label{int1conseq}
\begin{aligned}
\|\nabla \hat u \|_{L^2(B_{1/2})}^2 &\leq  K_1\left( \|\hat\Omega\|_{L^2(B_1)}^2 \|\nabla \hat u\|_{L^2(B_1)}^2 + \|\hat u\|_{L^1(B_1)}^2 + \|\hat f\|_{L^1(B_1)} \|\hat f\|_{L\ln L(B_1)}\right)\\
&\leq  K_1\|\hat\Omega\|_{L^2(B_1)}^2 \|\nabla \hat u\|_{L^2(B_1)}^2 + C\left(\|\hat u\|_{L^1(B_1)}^2 
+ \|\hat f\|_{L\ln L(B_1)}^2\right)\\
\end{aligned}\nonumber
\end{equation}
and (again by Section \ref{scalingsect}) this translates to
\begin{equation}
\begin{aligned}
\|\nabla u \|_{L^2(B_{R/2}(x_0))}^2 
&\leq  K_1\|\Omega\|_{L^2(B_R(x_0))}^2 \|\nabla u\|_{L^2(B_R(x_0))}^2 + C\left(R^{-4}\|u\|_{L^1(B_R(x_0))}^2 + \|f\|_{L\ln L(B_R(x_0))}^2\right).
\end{aligned}\nonumber
\end{equation}
Using our upper bound for $\eta$ and the fact that $R\leq 1$ we have, in particular
\begin{equation*}
 \|\nabla u \|_{L^2(B_{R/2}(x_0))}^2 \leq \epsilon_0 \|\nabla u\|_{L^2(B_R(x_0))}^2 + CR^{-4}\left(\|u\|_{L^1(B_1)}^2 + \|f\|_{L\ln L(B_1)}^2\right).
\end{equation*}
Letting $\Gamma = C\left(\|u\|_{L^1(B_1)}^2 + \|f\|_{L\ln L(B_1)}^2\right)$ we are precisely in the set-up of Lemma \ref{LS}, since this estimate is true in particular 
for all $B_{2R}(x_0) \subset B_1$. 
Therefore 
\begin{equation}
\label{W12est}
 \|\nabla u\|_{L^2(B_{1/2})} \leq C \left( \|u\|_{L^1(B_1)} + \|f\|_{L\ln L(B_1)} \right).
\end{equation}
It remains to improve this estimate to control the second
derivatives, and for that we use the second estimate \eqref{int2}
of the first part of Theorem \ref{mesest}, in the case 
$r=\frac 12$, which we then scale by a factor $\frac 12$ to
give:
\begin{equation*}
\label{int2conseq}
 \|\nabla^2 u \|_{L^1(B_{1/4})} \leq C\left(\|\Omega \|_{L^2(B_{1/2})} \|\nabla u \|_{L^2(B_{1/2})} + \|u\|_{L^1(B_{1/2})} + \|f\|_{L\ln L(B_{1/2})}\right).
\end{equation*}
Combining with \eqref{W12est} then yields
\begin{equation*}
 \|\nabla^2 u \|_{L^1(B_{1/4})} \leq 
 C\left(\|u\|_{L^1(B_1)} + \|f\|_{L\ln L(B_1)}\right),
\end{equation*}
and a simple rescaling and covering argument gives us that for any compactly contained $U\subset \subset B_1$ there is a $C=C(U, m) < \infty$ such that  
\begin{equation*}\label{nablaU}
 \|u\|_{W^{2,1}(U)} \leq C \left( \|u\|_{L^1(B_1)} + \|f\|_{L\ln L(B_1)} \right).
\end{equation*}

\section{Proof of the compactness, Theorem \ref{imp}}
\label{proof_imp}
Here we pick  $\eta_2 = \min \{\eta_1, \eta , \sqrt{\frac{1}{2K_2}}\}$ where $\eta_1$ is from Theorem \ref{betest}, and $\eta$ and $K_2$ are from the second part of Theorem \ref{mesest} for the choice $\delta = 1$.  We know (by Theorem \ref{betest}) that 
for all $U\subset\subset B_1$,
our sequence $\lbrace u_n \rbrace$ is uniformly bounded in $W^{2,1}(U)$, so by the Sobolev embedding theorem there exists some $u \in W^{1,2}_{loc}(B_1)$ such that (up to a subsequence) $u_n \rightharpoonup u$ weakly in $W^{1,2}(B_{2/3})$. We also know that $\lbrace \nabla u_n \rbrace$ is uniformly bounded in $W^{1,1}(B_{2/3})$, so by Lemma \ref{cmeasure} (with $ \nabla u_n  = V_n$) if we have
\begin{equation}
\label{no_conc}
\lim_{r\downarrow 0} \limsup _{n\rightarrow \infty} \|\nabla u_n \|_{L^2(B_{r}(x))} =0
\end{equation}
for all $x\in B_{2/3}$, then 
\begin{equation*}
 \nabla u_n \rightarrow \nabla u 
\end{equation*}
strongly in $L^{2}_{loc}(B_{2/3})$ which would prove the theorem. 
Therefore, it remains to prove \eqref{no_conc}.

Now pick $x_0 \in B_{2/3}$ and $R\in (0,1/2]$ small enough such that $B_R(x_0) \subset B_{2/3}$. 
Applying the second part of Theorem \ref{mesest} to the
rescaled scenario from Section \ref{scalingsect} 
(for each $n$) yields (for $r\in (0,1]$)
\begin{equation}
\begin{aligned}
\|\nabla \hat u_n \|_{L^2(B_r)}^2 &\leq  K_2\|\hat \Omega_n\|_{L^2(B_1)}^2 \|\nabla \hat u_n\|_{L^2(B_1)}^2 
 + 2r^2\|\nabla \hat u_n\|_{L^2(B_1)}^2 + K_2\|\hat f_n\|_{L^1(B_1)} \|\hat f_n\|_{L\ln L(B_1)}
\end{aligned}\nonumber
\end{equation}
and reversing the scaling leaves us with
\begin{equation}
\begin{aligned}
\|\nabla u_n \|_{L^2(B_{rR}(x_0))}^2 
&\leq
K_2\|\Omega_n\|_{L^2(B_R(x_0))}^2 
\|\nabla u_n\|_{L^2(B_R(x_0))}^2\\
&\quad + 2r^2 \|\nabla u_n\|_{L^2(B_R(x_0))}^2 + K_2\|f_n\|_{L^1(B_R(x_0))} \|f_n\|_{L\ln L(B_R(x_0))}\\
&\leq \frac{1}{2}\|\nabla u_n\|_{L^2(B_R(x_0))}^2\\
&\quad+ C \left(r^2 \|\nabla u_n\|_{L^2(B_R(x_0))}^2 + 
\left[\ln \left(\frac{1}{R}\right) \right]^{-1}\|f_n\|_{L\ln L(B_R(x_0))}^2\right)
\end{aligned}\nonumber
\end{equation}
using Lemma \ref{l1-llogl}.

Now, using that $\lbrace u_n \rbrace$ is uniformly bounded in $W^{2,1}(B_{2/3})$ and the hypotheses of the theorem, we have
\begin{equation*}
 \|\nabla u_n \|_{L^2(B_{rR}(x_0))}^2 \leq 1/2 \|\nabla u_n\|_{L^2(B_R(x_0))}^2 + C\left( r^2+ \left[\ln \left(\frac{1}{R}\right) \right]^{-1} \right)
\end{equation*}
Hence
\begin{equation*}
 \lim_{R\downarrow 0} \lim_{r \downarrow 0} \limsup_{n \rightarrow \infty}  \|\nabla u_n \|_{L^2(B_{rR}(x_0))}^2 \leq 1/2 \lim_{R\downarrow 0} \lim_{r \downarrow 0} \limsup_{n \rightarrow \infty}  \|\nabla u_n \|_{L^2(B_{R}(x_0))}^2
\end{equation*}
and we have shown that 
\begin{equation*}
 \lim_{r \downarrow 0} \limsup_{n \rightarrow \infty}  \|\nabla u_n \|_{L^2(B_{r}(x_0))}=0
\end{equation*}
which proves the theorem.

\section{Proof of the compact embedding 
$L\ln L\hookrightarrow H^{-1}$, Corollary \ref{compact_embedding}}
\label{proof_corollary}

In this section we use $\mathcal{H}^1-BMO$ duality  
(Appendix \ref{Hardy} and \cite{feff_stein}), the 
compactness result Theorem \ref{imp} 
and the continuous embedding
$W^{1,2}(\mathbb{R}^2) \hookrightarrow BMO(\mathbb{R}^2)$ 
(Appendix \ref{Lorentz}) to prove
the compactness of the embedding 
$L\ln L(B_1)\hookrightarrow H^{-1}(B_1)$.

First we check that the embedding $L\ln L\hookrightarrow H^{-1}$
exists and is continuous. We will realise $f\in L\ln L(B_1)$ as a bounded linear functional on $W_0^{1,2}(B_1)$. 

Recall from Appendix \ref{Lorentz} that if $f \in L\ln L(B_1)$ then $f-\bar f \in \mathcal{H}^1(\R^2)$ and $\|f -\bar f\|_{\mathcal{H}^1(\R^2)} \leq C\|f\|_{L\ln L(B_1)}$.  For $\phi \in W_0^{1,2}(B_1)$ we extend it by zero and calculate
\begin{eqnarray*}
\int_{B_1} f\phi &=& \int ( f -\bar f) \phi + \int  \bar f \phi \\
&\leq& C \|f- \bar f\|_{\mathcal{H}^1(\mathbb{R}^m)}\|\phi\|_{BMO(\mathbb{R}^2)} + \frac{1}{\pi} \|f\|_{L^1(B_1)}\|\phi\|_{L^1(B_1)} \\
&\leq& C \|f\|_{L\ln L(B_1)}\|\phi\|_{W_0^{1,2}(B_1)}
\end{eqnarray*}

Thus $f \in H^{-1}(B_1)$ and $\|f\|_{H^{-1}(B_1)} \leq C\|f\|_{L\ln L(B_1)}.$

Now consider a sequence $\{f_n\}\subset L\ln L(B_1)$ such that $\|f_n\|_{L\ln L(B_1)} \leq \Lambda < \infty$. We can extend each $f_n$ to be zero outside $B_1$ and consider the sequence of solutions $\{u_n\}\subset W_0^{1,2}(B_2)$ weakly solving 
\begin{equation*}
 -\Delta u_n = f_n \qquad \text{ on }B_2.
\end{equation*}
By the compactness of Theorem \ref{imp} we can conclude that there exists some $u \in W^{1,2}(B_2)$ 
such that (up to a subsequence) $u_n \rightarrow u$ strongly in $W^{1,2}(B_1)$. 

Writing $f=-\Delta u$ (which can clearly be viewed as an element of $H^{-1}(B_1)$) we see that 
\begin{eqnarray*}
 \|f_n - f\|_{H^{-1}(B_1)} &=& \sup_{\phi \in W_0^{1,2}(B_1)\,\,\,  \|\phi\|_{W_0^{1,2}(B_1) =1}} \int (f_n - f)\phi \\
&=& \sup_{\phi \in W_0^{1,2}(B_1)\,\,\,  \|\phi\|_{W_0^{1,2}(B_1) =1}}  \int (\nabla u_n - \nabla u).\nabla \phi \\
&\leq& \|\nabla u_n - \nabla u\|_{L^2(B_1)} \rightarrow 0
\end{eqnarray*}
 as $n \rightarrow \infty$. 

\section{Proof of the optimal regularity,
Theorem \ref{optimal_regularity}}
\label{proof_optimal_regularity}

The proof will proceed broadly in two steps.
First, we will use a type of `geometric bootstrapping'
to show that solutions $u$ have almost
the optimal regularity claimed.
\begin{lemma}
\label{almost_optimal_regularity}
Suppose $u\in W^{1,2}(B_1,\mathbb{R}^m)$ is a weak solution to 
\begin{equation*}\label{eqnin7.1}
-\Delta u = \Omega.\nabla u +f
\end{equation*}
where $\Omega \in L^2(B_1,so(m)\otimes \mathbb{R}^2)$
and $f\in L^p$ for some $p\in (1,2)$. Then 
$u\in W_{loc}^{2,\gamma}(B_1)$ for all $\gamma\in [1,p)$.
\end{lemma}
Once we have got this far, we can quantify the $W^{2,\gamma}$
regularity via the following estimate.
\begin{lemma}\label{lp}
Let $p\in(1,2)$ and
suppose $\gamma\in (\frac{p+1}{2},p)$ and 
$u \in W^{2,\gamma}(B_1, \R^m)$ solves 
\begin{equation*}%\label{eqn2}
-\Delta u = \Omega.\nabla u +f
\end{equation*}
where $\Omega \in L^2(B_1,so(m)\otimes \mathbb{R}^2)$
and $f\in L^p$. Then there exists $\eta_3 = \eta_3(p,m)  >0$
and $C= C(p,m) < \infty$ such that if 
$\|\Omega \|_{L^2(B_1)} \leq \eta_3$ then 
\begin{equation}
\label{alphaest}
\|u\|_{W^{2, \gamma}(B_{1/2})} \leq C(\| f\|_{L^p(B_1)} + \|u\|_{L^1(B_1)}).
\end{equation}
\end{lemma}

Let us assume these two lemmata for the moment and see how 
Theorem \ref{optimal_regularity} follows.
First, Lemma \ref{almost_optimal_regularity} directly applies,
and we deduce $W_{loc}^{2,\gamma}(B_1)$ regularity for 
$\gamma\in [1,p)$.
If we assume that 
$\|\Omega \|_{L^2(B_1)} \leq \eta_3$, then 
\eqref{alphaest} of Lemma \ref{lp} holds
for $\gamma\in (\frac{p+1}{2},p)$
(strictly speaking we should make a small rescaling so that 
we can assume that $u\in W^{2,\gamma}(B_1)$ for 
$\gamma\in [1,p)$)
and by taking the limit $\gamma\nearrow p$, we deduce that
\begin{equation*}
\|u\|_{W^{2, p}(B_{1/2})} \leq 
C (\| f\|_{L^p(B_1)} + \|u\|_{L^1(B_1)}).
\end{equation*}
By an appropriate covering argument, working on balls
small enough so that $\|\Omega \|_{L^2}\leq \eta_3$,
we deduce the regularity claimed in the theorem, and
the claimed estimate \eqref{w2pest}.

It remains therefore to prove Lemmata \ref{almost_optimal_regularity} and \ref{lp}. We will need, in turn,
an additional lemma, which expresses the decay of energy
of solutions $u$.

\begin{lemma}
\label{decay_app}
Under the hypotheses of Lemma \ref{almost_optimal_regularity},
set $\alpha=2(1-1/p)\in (0,1)$.
Then there exists $\eta_4>0$ depending on $p$ and $m$ such that
if $\|\Omega\|_{L^2(B_1)}\leq \eta_4$ then
$$\sup_{x_0\in B_{1/2}, r\in (0,1/2)}
r^{-2\alpha}\|\nabla u\|_{L^2(B_r)}^2<\infty,$$
and in particular, so that
$$\sup_{x_0\in B_{1/2}, r\in (0,1/2)}
r^{-\alpha}
\|\Omega.\nabla u\|_{L^1(B_r)}<\infty.$$
\end{lemma}

\begin{proof} (Lemma \ref{decay_app}.)
For reasons that will become apparent, choose $\delta\in (0,1]$
sufficiently small so that 
\begin{equation}
\label{delta_small}
\lambda:=\frac{(1+2\delta)}{4}<2^{-4(1-1/p)}=:\Lambda\in \left(\frac{1}{4},1\right)
\end{equation}
We can now choose 
$\eta_4:=\min\{\eta,\sqrt{\frac{\delta}{4K_2}}\}$, where
$\eta$ is from the
second part of Theorem \ref{mesest}, depending on the $\delta$
we have just chosen and therefore on $p$ (as well as $m$) and 
where $K_2$ is also from the second part of
Theorem \ref{mesest}.
Now take an arbitrary point $x_0\in B_{1/2}$ and any
$R\in (0,1/2)$.
Estimate \eqref{est2} in the case that $r=\frac{1}{2}$,
applied to the rescaled quantities defined in Section 
\ref{scalingsect} yields
\begin{equation}
\begin{aligned}
\|\nabla \hat u \|_{L^2(B_{1/2})}^2 &\leq  
\frac{(1+\delta)}{4} \|\nabla \hat u\|_{L^2(B_1)}^2 +
K_2\left( \|\hat\Omega\|_{L^2(B_1)}^2 \|\nabla \hat u\|_{L^2(B_1)}^2 +  \|\hat f\|_{L^1(B_1)} \|\hat f\|_{L\ln L(B_1)}\right)\\
&\leq
\lambda \|\nabla \hat u\|_{L^2(B_1)}^2 +
C\|\hat f\|_{L^p(B_1)}^2.
\end{aligned}\nonumber
\end{equation}
Reversing the scaling, using Section \ref{scalingsect},
we find that
\begin{equation}
\begin{aligned}
\|\nabla u \|_{L^2(B_{R/2}(x_0))}^2 &\leq  
\lambda \|\nabla u\|_{L^2(B_R(x_0))}^2 +
CR^{4(1-1/p)}\|f\|_{L^{p}(B_R(x_0))}^2\\
&\leq 
\lambda \|\nabla u\|_{L^2(B_R(x_0))}^2 +
[C\|f\|_{L^{p}(B_1)}^2] R^{4(1-1/p)}.
\end{aligned}\nonumber
\end{equation}
Now applying what we have proved
for $R=2^{-k}$, with $k\in \{1,2\ldots\}$
and using the abbreviation
$a_k:=\|\nabla u \|_{L^2(B_{2^{-k}}(x_0))}^2$, we find 
that
\begin{equation}
\begin{aligned}
a_{k+1}&\leq \lambda a_k + K_3 2^{-4(1-1/p)k}\\
&= \lambda a_k + K_3 \Lambda^k
\end{aligned}\nonumber
\end{equation}
where $K_3$ is independent of $x_0$.
This recursion relation can be solved to yield
$$a_{k+1}\leq \lambda^k a_1
+K_3\Lambda\frac{(\Lambda^k-\lambda^k)}{\Lambda-\lambda},$$
and by \eqref{delta_small}, this simplifies to
$$\|\nabla u \|_{L^2(B_{2^{-k}}(x_0))}^2=:
a_k\leq C\Lambda^k.$$
Thus, for $r\in (0,1/2]$ we have
$$\|\nabla u \|_{L^2(B_r(x_0))}^2\leq Cr^{4(1-1/p)},$$
and hence the lemma is proved.
\end{proof}

\begin{proof} (Lemma \ref{almost_optimal_regularity}.)
Our goal is to prove $W^{2,\gamma}$ regularity, for all
$\gamma\in [1,p)$.
By applying
Calderon-Zygmund theory
directly to the equation \eqref{eqn2},
this will follow if we can show that 
$\Omega.\nabla u \in L^s$ for all $s\in [1,p)$.

By the nature of what we are trying to prove, we may also 
assume that $\|\Omega\|_{L^2(B_1)}\leq\eta_4$, where
$\eta_4$ is from Lemma \ref{decay_app}. If this is not true,
would could apply the result, after rescaling, on appropriate 
small balls where it is true.

By Lemma \ref{decay_app}, and the first part of 
Lemma \ref{adams_lemma} with
$h$ chosen to be $\Omega.\nabla u$, we can deduce that for any 
$q\in (1,\frac{2-\alpha}{1-\alpha})=(1,\frac{2}{2-p})$, 
we have 
$$I_1h\in L^q(B_{1/4}),$$
and in particular, this implies that 
$u\in W^{1,q}(B_{1/4})$ for some $q>2$. 
If we apply the same result over appropriate smaller balls,
we have in fact that $u\in W_{loc}^{1,q}(B_1)$ for some $q>2$.
At that point, we know 
that $\Omega.\nabla u\in L_{loc}^s(B_1)$ with 
$s=\frac{2q}{2+q}\in (1,\frac{2}{3-p})\subset(1,p)$.

We will now show that from here it is possible to carry out a 
geometric bootstrapping argument by using the second part of Lemma \ref{adams_lemma}.
As an aside, we note that a bootstrapping argument
using only classical Calderon-Zygmund methods does not work.
The bootstrapping claim is that 
\begin{equation}
\label{bootstrap}
\Omega.\nabla u\in L_{loc}^s(B_1)\text{ with }s \in (1,p)
\implies \Omega.\nabla u\in L_{loc}^{s(\frac{2}{2+s-p})}(B_1).
\end{equation}
This is true because whenever we know that
$\Omega.\nabla u\in L_{loc}^s(B_1)$ with 
$s \in (1,p)$, then 
we can deduce from Lemma \ref{decay_app}, and the second part of 
Lemma \ref{adams_lemma} with $h$ chosen to be $\Omega.\nabla u$,
that 
$$I_1h\in L^{s\frac{2-\alpha}{1-\alpha}}(B_{1/4}),$$
and recalling that 
$\frac{2-\alpha}{1-\alpha}=\frac{2}{2-p}$,
this is enough to establish this time that
$u\in W^{1,\frac{2s}{2-p}}(B_{1/4})$ 
and hence
$\Omega.\nabla u\in L^{s(\frac{2}{2+s-p})}(B_{1/4})$.
Again, by a simple covering argument, we deduce that
$\Omega.\nabla u\in L_{loc}^{s(\frac{2}{2+s-p})}(B_1)$
as desired.

By iterating the bootstrapping claim \eqref{bootstrap},
we find that 
$\Omega.\nabla u\in L_{loc}^s(B_1)$
for all $s \in [1,p)$, and the proof is complete.
\end{proof}

It remains to prove Lemma \ref{lp}.

\begin{proof} (Lemma \ref{lp}.)
We begin by applying the Calderon-Zygmund estimate from Lemma \ref{CZdp}, giving some $C$ independent of $\gamma$ such that
\begin{equation*}
\label{CZestimate}
\|u\|_{W^{2,\gamma}(B_{1/2})}\leq \frac{C}{\gamma -1}(\|\Delta u\|_{L^\gamma(B_{2/3})}+\|u\|_{L^\gamma(B_{2/3})}),
\end{equation*}
valid for $\gamma\in (1,2]$ and $u\in W^{2,\gamma}(B_1)$.

For the specific $u$ of the lemma, and 
$\gamma\in (\frac{1+p}{2},p)$, we may then compute
(also using the inequalities of Sobolev and H\"older,
and Lemma \ref{nablauL})
\begin{eqnarray}\label{CZ}
\|u\|_{W^{2,\gamma}(B_{1/2})} 
&\leq &C (\|\Omega .\nabla u\|_{L^\gamma(B_{2/3})}
 + \|f\|_{L^\gamma(B_{2/3})} + \|u\|_{L^2(B_{2/3})})\nonumber \\ 
&\leq& C (\|\Omega .\nabla u\|_{L^\gamma(B_{2/3})} + \|f\|_{L^\gamma(B_{2/3})} + \|u\|_{W^{1,1}(B_{2/3})})\nonumber \\ 
&\leq& C (\|\Omega \|_{L^2(B_1)} \|\nabla u\|_{L^{\frac{2\gamma}{2-\gamma}}(B_1)} + \|f\|_{L^\gamma(B_1)} + \|u\|_{L^1(B_1)}).
\end{eqnarray}
where $C$ depends only on $p$ and $m$.
Now consider any $v \in W^{2,t}(B_{1/2})$ for $t\in [1, 2)$. 
We can find a $t$-independent extension operator 
$Ex: W^{2,t}(B_{1/2}) \rightarrow W^{2,t}(\R^2)$ 
whose images have compact support in $B_1$, so that there 
exists some $C<\infty$ independent of $t\in [1,2)$ 
such that (denoting $Ex(v) = \tilde v$)
\begin{equation}\label{Ex}
\|\tilde v\|_{W^{2,t}(\R^2)} \leq C \|v\|_{W^{2,t}(B_{1/2})}
\end{equation}
with $v=\tilde v$ in $B_{1/2}$. 
(See for instance \cite[Theorem 7.25]{gt}.)

From here we apply the standard Sobolev embedding for $\nabla \tilde v \in W^{1,t} \hookrightarrow L^{\frac{2t}{2-t}}$ to obtain (see \cite[Theorem 7.10]{gt})
\begin{equation*}
\|\nabla \tilde v\|_{L^{\frac{2t}{2-t}}} \leq \frac{t}{2-t} \|\nabla^2 \tilde v \|_{L^t}.
\end{equation*}

This, coupled with (\ref{Ex}) gives us 
\begin{equation}\label{Sobolev}
\|\nabla v\|_{L^{\frac{2t}{2-t}}(B_{1/2})} \leq \frac{C}{2-t} \|v\|_{W^{2,t}(B_{1/2})}
\end{equation}
for all $v \in W^{2,t}(B_{1/2})$, $t\in [1, 2)$ and where $C$ is independent of $t$. 

Using (\ref{CZ}) and (\ref{Sobolev}) we have that there exists some $C  < \infty$ depending only on $p$ and $m$ such that 
\begin{equation}\label{improvedCZ}
\|\nabla u\|_{L^{\frac{2\gamma}{2-\gamma}}(B_{1/2})} \leq  C(\|\Omega \|_{L^2(B_1)} \|\nabla u\|_{L^{\frac{2\gamma}{2-\gamma}}(B_1)} + \|f\|_{L^\gamma(B_1)} + \|u\|_{L^1(B_1)})
\end{equation}
for all $\gamma\in (\frac{1+p}{2} ,p)$.

We now choose $B_R(x_0) \subset B_1$; using the scaling
of Section \ref{scalingsect}, the estimate \eqref{improvedCZ}
yields
\begin{equation*}
\|\nabla \hat u\|_{L^{\frac{2\gamma}{2-\gamma}}(B_{1/2})} \leq  C(\|\hat{\Omega} \|_{L^2(B_1)} \|\nabla  \hat u\|_{L^{\frac{2\gamma}{2-\gamma}}(B_1)} + \|\hat f\|_{L^\gamma(B_1)} + \|\hat u\|_{L^1(B_1)})
\end{equation*}
which translates (using $R\leq 1$ and Section \ref{scalingsect})
to 
\begin{equation*}
\|\nabla  u\|_{L^{\frac{2\gamma}{2-\gamma}}(B_{R/2}(x_0))} \leq  C(\|\Omega \|_{L^2(B_R(x_0))} \|\nabla u\|_{L^{\frac{2\gamma}{2-\gamma}}(B_R(x_0))}
 + R^{-3}(\| f\|_{L^\gamma(B_R(x_0))} +\|u\|_{L^1(B_R(x_0))})).
\end{equation*}

Now let $\beta= \frac{2\gamma}{2-\gamma}$ and raise this inequality to the power $\beta$ to obtain (noticing $\beta < \frac{2p}{2-p}$)
\begin{eqnarray*}
\|\nabla  u\|_{L^{\beta}(B_{R/2}(x_0))}^{\beta} &\leq&  K\|\Omega \|_{L^2(B_R(x_0))}^{\beta} \|\nabla u\|_{L^{\beta}(B_R(x_0))}^{\beta}\\
 &+& R^{-\frac{6p}{2-p}}K \left(\| f\|_{L^\gamma(B_R(x_0))} + \|u\|_{L^1(B_R(x_0))})\right)^{\beta}
\end{eqnarray*}
for some specific $K<\infty$ depending only on $p$ and $m$.

We now wish to apply Lemma \ref{LS}: We are able to choose $\eta_3 = \left(\frac{\epsilon_0}{K}\right)^{\frac{1}{\beta}}$ where $\epsilon_0$ is that of Lemma \ref{LS} corresponding to the choice $k = \frac{6p}{2-p}$. Let $\Gamma =K \left( \| f\|_{L^\gamma(B_1)} + \|u\|_{L^1(B_1)})\right)^{\beta}$.  

The above estimate holds in particular for any $B_R(x_0)$ such that $B_{2R}(x_0)\subset B_1$, so we have
\begin{equation*}
\|\nabla  u\|_{L^{\beta}(B_{R/2}(x_0))}^{\beta} \leq \epsilon_0 \|\nabla u\|_{L^{\beta}(B_R(x_0))}^{\beta} +R^{-\frac{6p}{2-p}}\Gamma.
\end{equation*}
Hence a direct application of Lemma \ref{LS} gives us 
\begin{equation*}
\|\nabla  u\|_{L^{\frac{2\gamma}{2-\gamma}}(B_{1/2})} \leq C(\| f\|_{L^{p}(B_1)} + \|u\|_{L^1(B_1)}).
\end{equation*}
Therefore a simple covering argument yields that there is a $C=C(p,m) < \infty$ such that 
\begin{equation*}
\|\nabla  u\|_{L^{\frac{2\gamma}{2-\gamma}}(B_{2/3})} \leq C (\| f\|_{L^{p}(B_1)} + \|u\|_{L^1(B_1)}).
\end{equation*}
and from here another application of Lemmata \ref{CZdp} and \ref{nablauL}
tells us that there is a $C=C(p,m) < \infty$ such that
\begin{equation*}
\|u\|_{W^{2,\gamma}(B_{1/2})} \leq C (\| f\|_{L^{p}(B_1)} + \|u\|_{L^1(B_1)}).
\end{equation*}
\end{proof}

\section{$L\ln L$ cannot be replaced by $h^1$}
\label{H1doesntwork}
Here we present a counterexample to the compactness Theorem \ref{imp} when we allow $f_n \in h^1(B_1)$. 
Our example will have $\Omega_n \equiv 0$ for all $n$ and $u_n : B_1 \rightarrow \mathbb{S}^2$ will be a sequence of harmonic maps with bounded energy that undergoes bubbling. 

Let  $\pi : \R^2 \rightarrow \mathbb{S}^2$ be the (inverse of) stereographic projection and take $u_n (x,y) = \pi (nx, ny)$.
Since $u_n$ is harmonic for all $n$ we know it solves (see \cite{helein_conservation})
 \begin{equation*}\label{un}
 -\Delta u_n = \nabla^{\bot} B_n . \nabla u_n 
 \end{equation*}
 where $\nabla^{\bot} (B_n)^i_j= (u_n)^i\nabla (u_n)^j - (u_n)^j\nabla (u_n)^i$, therefore $\|\nabla^{\bot}B_n\|_{L^2(\R^2)} \leq C\|\nabla u_n\|_{L^2(\R^2)}$.
 Letting $f_n = \nabla^{\bot} B_n . \nabla u_n $ we have 
$$\|f_n\|_{\mathcal{H}^1(\R^2)} \leq C\|\nabla u_n \|_{L^2(\R^2)}^2 \leq C\|\nabla \pi \|_{L^2(\R^2)}^2 = \Lambda <\infty.$$ 
This tells us two things: first that $\|f_n\|_{h^1(B_1)} \leq \Lambda < \infty$ and second that $\|u_n\|_{W^{1,2}(B_1)} \leq \Lambda < \infty$. At this point we have all the hypotheses of the theorem (except that we allow $f_n \in h^1$), but it is easy to see that there can be no subsequence converging locally strongly in $W^{1,2}$ because this sequence forms a bubble at the origin.

%%%%%%%%%%%%%%%%%%%%%%%%%%%%%%%%%%%%%%%%%%%%%%%%%%%%%%

\appendix
\section{Background and supporting results}

\subsection{Singular Integrals}\label{sing}

We recall here the basics of Calderon-Zygmund theory 
on the unit ball $B_1\subset \R^2$, following 
\cite{gt}. 
Define the Newtonian potential operator $N$ on functions $f\in L^1(B_1)$
(implicitly extended to be zero to $\mathbb{R}^2\backslash B_1$) by 
\begin{equation*}
 N[f](x):= (\Gamma \ast f) (x) = \int \Gamma(x-y) f(y) \,dy 
\end{equation*}
where $\Gamma (x) =\frac{1}{2\pi}\ln  |x|$.
If $f\in C_c^{\infty}$ then $\Delta N[f] = f$.
Writing $w=N[f]$ we have 
\begin{enumerate}
\item 
$N : L^p(B_1) \rightarrow L^p(B_1)$ is a bounded operator for all $1\leq p \leq \infty$.
\item
$\nabla w = \nabla N[f] = (\nabla \Gamma) \ast f$. We will frequently view $\nabla N$ as an operator in its own right. It is easy to see that $|\nabla N[f]| \leq \frac{1}{2\pi}I_1[|f|]$ where $I_1$ is the standard notation for this Riesz potential (defined by convolution with $\frac{1}{|x|}$) and $I_1 : L^p(B_1) \rightarrow L^{\frac{2p}{2-p}}(B_1)$ is a bounded operator for all $1<p<2$. 
\item (Calderon-Zygmund)
Let $f \in L^p(B_1)$, $1<p<\infty$ and $w=N[f]$. Then $\nabla^2 w = \nabla^2 N[f] = (\nabla^2 \Gamma) \ast f$ and (as above we see $\nabla^2 N$ as an operator) $\nabla^2 N : L^p(B_1) \rightarrow L^p(B_1)$ is bounded for $p$ in this range. 
More explicitly we have $w \in W^{2,p}(B_1)$, $\Delta w = f$ almost everywhere, and 
\begin{equation}\label{CZ1}
\|\nabla^2 w\|_{L^p(B_1)} \leq C(p) \|f\|_{L^p(B_1)}.
\end{equation}
\end{enumerate}
In fact, revisiting a second time the proof of the 
Calderon-Zygmund estimates (e.g. \cite[\S 9.4]{gt}, but
interpolating between $q=1$ and $r=3$) we find that
the dependency of $C$ in \eqref{CZ1} can be weakened,
and one can prove: 

\begin{lemma}\label{CZdp}
Let $\gamma \in (1,2]$ and suppose $u \in W^{2,\gamma}(B_1)$. Then there exists some $C< \infty$ independent of $\gamma$ such that
\begin{equation*}
\|u\|_{W^{2,\gamma}(B_{1/2})} \leq \frac{C}{\gamma -1}(\|\Delta u\|_{L^{\gamma}(B_{2/3})} + \|u\|_{L^{\gamma}(B_{2/3})})
\end{equation*}
\end{lemma}

Of course, even in the $L^1$ case, we have sub-optimal estimates
such as:

\begin{lemma}\label{nablauL}
If $u\in L^1(B_1)$ is a weak solution to 
\begin{equation*}
 -\Delta u  =f\in L^1(B_1),
\end{equation*}
then $u\in W^{1,1}(B_{2/3})$ and there exists $C<\infty$ such that
\begin{equation*}
 \|\nabla u \|_{L^1(B_{2/3})} \leq C ( \|f \|_{L^1(B_1)} + \|u\|_{L^1(B_1)}).
\end{equation*}
\end{lemma}

We now present a theorem of Adams giving improved estimates on the Riesz potential $I_1$ if, in addition we have a decay estimate on our function.
Given $h\in L^1(B_1,\R^m)$ define a new function $I_1 h$ to be 
the convolution of $|h|$, extended to be zero on 
$\R^2\backslash B_1$, with $|x|^{-1}$, i.e.
$$I_1 h(x)=\int_{B_1}\frac{|h(y)|}{|x-y|}dy.$$

\begin{lemma} (Adams \cite[Propositions 3.1 and 3.2]{adams_riesz}.)
\label{adams_lemma}
Suppose $h\in L^1(B_1)$ and
$$\sup_{x\in B_{1/2}, r\in (0,1/2)}
r^{-\alpha}\int_{B_r(x)}|h|<\infty,$$
for some $\alpha\in (0,1)$.
Then 
\begin{enumerate}
\item
$I_1h\in L^q(B_{1/4})$
for any $q\in [1,\frac{2-\alpha}{1-\alpha})$;
\item
if in addition $h\in L_{loc}^s(B_1)$ for $s\in (1,\frac{2}{2-\alpha})$, 
then $I_1h\in L^{s\left(\frac{2-\alpha}{1-\alpha}\right)}(B_{1/4})$.
\end{enumerate}
\end{lemma}

\subsection{Hardy Spaces}\label{Hardy}

Pick $\phi \in C_c^{\infty}(B_1)$ such that $\int \phi = 1$ and let $\phi_t (x) = t^{-2}\phi(\frac{x}{t})$. For a distribution $f$ we say $f$ lies in the Hardy space $\mathcal{H}^1(\R^2)$ if $f_{\ast} \in L^1(\R^2)$ where 

\begin{equation*}
f_{\ast}(x) = \sup_{t>0} |(\phi_t \ast f ) (x)|
\end{equation*}
with norm $\|f\|_{\mathcal{H}^1(\R^2)}= \|f_{\ast}\|_{L^1(\R^2)}$. Clearly we have the continuous embedding $\mathcal{H}^1(\R^2) \hookrightarrow L^1(\R^2)$. 
The dual space of $\mathcal{H}^1(\R^2)$ is $BMO(\R^2)$ where $BMO:=\{ g \in L^1_{loc}(\R^2) : \sup_{B \subset \R^2} \frac{1}{|B|} \int_B |g - \bar g| < \infty \}$ (see \cite{feff_stein}). 

Related to $\mathcal{H}^1$ is the so-called local Hardy space $h^1$ defined to be those functions for which
 $$f_{\tilde \ast}(x) = \sup_{0<t<1} |(\phi_t \ast f ) (x)| \in L^1(\R^2)$$ 
with corresponding norm. 
Again we clearly have the continuous embedding $h^1(\R^2) \hookrightarrow L^1(\R^2)$. 
For a function $f$ defined in $B_1$ we say that $f \in h^1(B_1)$ if
$$f_{\hat \ast}(x) = \sup_{0<t<1-|x|} |(\phi_t \ast f ) (x)| \in L^1(B_1).$$
By \cite[Theorem 1.92]{semmes_primer} we know that $f \in h^1(B_1)$ if and only if for any $\varphi \in C_c^{\infty}(B_1)$ with $\int \varphi \neq 0$ there is a constant $\lambda$ such that 
$\varphi(f-\lambda) \in \mathcal{H}^1(\R^2)$,
with 
$$\|\varphi(f-\lambda)\|_{\mathcal{H}^1(\R^2)} \leq C \|f\|_{h^1(B_1)},$$
where $C = C(\varphi)$ and $\lambda$ is chosen such that $\int \varphi(f-\lambda) =0$.

The Hardy spaces act as replacements to $L^1$ in Calderon-Zygmund estimates. In particular for $f \in \mathcal{H}^1(\R^2)$,
writing $w=N[f]$, we have the estimate (see \cite{helein_conservation})
\begin{equation}\label{HCZ}
\|\nabla^2 w\|_{L^1(\R^2)} \leq C \|f\|_{\mathcal{H}^1(\R^2)},
\end{equation} 
and if $f \in h^1(\R^2)$ then $\nabla^2 w \in L^1_{loc}(\R^2)$ with
$$\|\nabla^2 w\|_{L^1(B_1)} \leq C\|f\|_{h^1(\R^2)}.$$
Moreover $f \in h^1(B_1)$ implies $\nabla^2 w \in L^1_{loc}(B_1)$. The final two assertions follow from \eqref{HCZ}, \cite[Theorem 1.92]{semmes_primer} and standard elliptic estimates.

\subsection{Lorentz Spaces and $L\ln L$}\label{Lorentz}

For measurable $f$ define, for $s\geq 0$, the distribution function
$\lambda(s)=|\{x : |f|(x) >s\}|$. Assuming $\lim_{s\to\infty}\lambda(s)=0$,
define the nonincreasing rearrangement $f^{\ast}: (0,\infty)\rightarrow [0,\infty)$ by
\begin{equation*}
f^{\ast}(t):= \inf \{s\geq 0 : \lambda(s)\leq t\}.
\end{equation*}
Here we consider the spaces defined by:
\begin{enumerate}
\item
$L^{2,1} := \{f: \int t^{-1/2}f^{\ast}(t)\,\, dt < \infty\}$
\item
$L^{2,\infty} := \{f : \sup_{t>0} t^{1/2} f^{\ast}(t) <\infty \}$
\item
$L\ln L := \{ f : \int f^{\ast}(t) \ln \left( 2+ \frac{1}{t} \right) \,\,dt < \infty \}$
\end{enumerate}
The quantities above are not norms, but the spaces are all Banach spaces whose norms are equivalent to these quantities respectively.  The spaces $L^{2,1}$ and $L^{2,\infty}$ are two examples of Lorentz spaces, which can be thought of as perturbations of the usual $L^p$ spaces. For example the following are all continuous embeddings (see \cite{ziemer})
\begin{equation*}
L^p(B_1) \hookrightarrow L^{2,1}(B_1)  \hookrightarrow L^2(B_1)=L^{2,2}(B_1) \hookrightarrow L^{2,\infty}(B_1) \hookrightarrow L^q(B_1)
\end{equation*}
for all $q<2<p$. 
The dual space of $L^{2,1}$ is 
$L^{2,\infty}$ \cite{hunt}.

For the space $L\ln L$ we have the continuous embeddings 
\begin{equation*}
L^p(B_1) \hookrightarrow L\ln L(B_1)  \hookrightarrow L^1(B_1)
\end{equation*}
for all $p>1$. It is well known \cite{stein_llogl} that $f \in L\ln L$ if and only if its corresponding maximal function is locally integrable, where the maximal function $M_0(f) (x) = \sup_{t>0} \frac{1}{|B_t(x)|}\int_{B_t(x)} |f|$. By extension of $f$ by zero  and comparison of the functions $M_0 (f)$ and $f_{\tilde \ast}$ we see that the following embedding is continuous 
\begin{equation*}
L\ln L(B_1)  \hookrightarrow h^1(\R^2).
\end{equation*}
In fact, if $f \in L \ln L(B_1)$ then $f-\bar f \in \mathcal{H}^1(\R^2)$ with $\|f-\bar f\|_{\mathcal{H}^1(\R^2)} \leq C \|f\|_{h^1(\R^2)} \leq C \|f\|_{L\ln L(B_1)}$, where
\[ f - \bar f  := \left\{ \begin{array}{cl}
         f-\frac{1}{|B_1|}\int_{B_1}f & \text{in }B_1;\\
        0 & \text{in }\R^2\backslash B_1. \end{array} \right. \]

\subsection{Embeddings and Estimates}\label{misc}

Listed below are some important miscellaneous results involving Sobolev spaces and the spaces mentioned above. 
\begin{enumerate}
\item
The embedding $W^{1,1}(\R^2) \hookrightarrow L^{2,1}(\R^2)$ is continuous \cite{helein_conservation}.
\item
Given $u,v \in W^{1,2}(\R^2)$ then $\nabla u .\nabla^{\bot} v \in \mathcal{H}^1(\R^2)$ with \newline $\|\nabla u .\nabla^{\bot} v\|_{\mathcal{H}^1(\R^2)} \leq C\|\nabla u \|_{L^2(\R^2)}\|\nabla v\|_{L^2(\R^2)}$. (See \cite{clms}.)
\item
By 1. and the estimates from Appendices \ref{sing} and \ref{Hardy} on the Newtonian potential we have that the operators $\nabla N : h^1(\R^2) \rightarrow L_{loc}^{2,1}(\R^2)$ and $\nabla N : \mathcal{H}^1(\R^2) \rightarrow L^{2,1}_{loc}(\R^2)$ are bounded. 
\item
$\nabla N : L^1(B_1) \rightarrow L^{2,\infty}(B_1)$ is a bounded operator; this follows by standard estimates on convolutions and the fact that $\nabla \Gamma \in L^{2,\infty}$.
\item
The embedding $W^{1,2}(\R^2) \hookrightarrow BMO(\R^2)$ is continuous.
\end{enumerate}

\subsection{Riviere's gauge}

The key result from Rivi\`ere's work that we will need is the 
existence of the following perturbation of Coulomb's 
gauge.

\begin{lemma}[Rivi\`{e}re]\label{Riv}
 Suppose $\Omega \in L^2(B_1,so(m)\otimes \mathbb{R}^2)$. Then there exists $\epsilon =\epsilon(m)>0$ such that if $\|\Omega\|_{L^2(B_1)} \leq \epsilon$ we can find $A\in W^{1,2}(B_1, GL_m (\mathbb{R}))\cap L^{\infty} (B_1, GL_m (\mathbb{R}))$, $B\in W^{1,2}(B_1, gl_m (\mathbb{R}))$ and $C=C(m)<\infty$ where
\begin{equation*}
 \nabla A - A\Omega = \nabla^{\bot} B
\end{equation*}
and 
\begin{equation*}
 \|\nabla A\|_{L^2(B_1)} + \|\nabla B\|_{L^2(B_1)} + \|dist(A,SO(m)\|_{L^{\infty}(B_1)} \leq C \|\Omega\|_{L^2(B_1)}.
\end{equation*}

\end{lemma}

\subsection{Weak Convergence of Measures and Functions of Bounded Variation}\label{BV}

We consider the space of functions of bounded variation $BV(B)$ for any ball $B \subset \R^2$. $BV$ is defined by 
$BV(B) = \{ V \in L^1(B) : \int_B |\nabla V| :=  \sup_{\phi \in C_0^1 (B, \R^2) \,\,\,\|\phi\|_{L^{\infty}}\leq 1} \int_B V \text{div} \phi  < \infty \}$. In other words it is the space of functions whose distributional derivatives are signed Radon measures with finite total mass. This is a Banach space with norm $\|V\|_{BV(B)} = \|V\|_{L^1(B)} + \int_B |\nabla V|$. It is easy to see that we have the continuous embedding $W^{1,1} \hookrightarrow BV$, moreover we have the continuous embedding $BV(B) \hookrightarrow L^2(B)$ and the compact embeddings $BV(B) \hookrightarrow L^p(B)$ for any $p<2$ (see for instance \cite{ziemer}).

We also use the standard weak-$\ast$ compactness available in the space of signed Radon measures with finite total mass, denoted $M$. 

The proof of the next lemma is essentially taken from \cite[Theorem 9]{evans_weak_convergence} and is similar to that stated in \cite{LZ}. 
For an integrable function $k$ we implicitly view it as both a function and a measure, i.e. $k = k \,dx$. 

\begin{lemma}\label{measure}
Suppose  $\lbrace V_n \rbrace \subset BV(B)$ is a bounded sequence and $B\subset \mathbb{R}^2$ is an open ball. Then there exist at most countable $\lbrace x_j \rbrace \subset B$ and $\lbrace a_j >0 \rbrace$ (where $\sum_j a_j <\infty$) and $V\in BV(B)$ such that (up to a subsequence)
\begin{equation*}
 V_n^2 \rightharpoonup V^2 + \sum_j a_j \delta_{x_j}
\end{equation*}
weakly in $M(B)$.
\end{lemma}

\begin{proof}
Since $\lbrace V_n \rbrace \subset BV(B)$ is a bounded sequence, there exists $V \in L^2$ such that (up to a subsequence) $V_n \rightarrow V$ strongly in $L^p$ for all $p<2$ and $V_n \rightharpoonup V$ weakly in $L^2$. 
Also $\{ \nabla V_n \} \subset M(B)$ is bounded so (again up to a subsequence) $\nabla V_n \rightharpoonup \lambda$ (a vector-valued measure)$ \in M(B)$. In particular, for all $\phi \in C_{c}^{1}(B,\R^2)$ 
\begin{eqnarray*}
\int \phi. \, d\lambda &=& \lim_{n\rightarrow \infty} \int \phi. \nabla V_n \,dx \\
&=& -\lim_{n\rightarrow \infty} \int \text{div}(\phi) V_n \, dx \\
&=& - \int \text{div}(\phi) V \, dx.
\end{eqnarray*}
In other words $V\in BV(B)$ and $\nabla V = \lambda$.

Now set $g_n := V_n - V$. Note that $|\nabla g_n | \in M(B)$ is bounded so for a subsequence $|\nabla g_n | \rightharpoonup \mu \in M(B)$ where $\mu$ is non-negative. 
Similarly (up to a subsequence) $g_n^2 \rightharpoonup \nu \in M(B)$ where $\nu$ is also non-negative. 
We have that for all $\phi \in C_{c}^{1}(B)$, $\phi g_n \in BV(B)$ and by the continuous embedding $BV(B)\hookrightarrow L^2(B)$ we have
\begin{equation*}
 \left( \int (\phi g_n)^2 \, dx \right)^{1/2} \leq C \int |\nabla(\phi g_n)| \, dx
\end{equation*}
and since $g_n \rightarrow 0$ in $L^1$, taking limits gives 
\begin{equation*}
 \left( \int \phi^2 \, d\nu \right)^{1/2} \leq C \int |\phi| \, d\mu.
\end{equation*}
Taking $\phi$ to be an approximation to the characteristic function on  $B_r(x) \subset B$ we get
\begin{equation*}
 \nu (B_r(x)) \leq C(\mu(B_r(x)))^2
\end{equation*}
for all $B_r(x) \subset B$, and in particular $\nu \ll \mu$. 

By standard results for differentiation of measures (see e.g. \cite[\S 1.6 Theorem 2]{evans_gariepy}), for any Borel set $E \subset B$
\begin{equation*}
 \nu(E) = \int_E D_{\mu} \nu \, d\mu 
\end{equation*}
where $D_{\mu} \nu = \lim_{r\downarrow 0} \frac{\nu (B_r(x))}{\mu(B_r(x))}$ is a $\mu$-integrable function (this limit exists $\mu$-almost everywhere). 

Since $\mu$ is a finite, positive Radon measure, there are at most countable points $\{x_j\}$ such that $\mu (\{x_j\}) >0$, and if $\mu (\{x\}) =0$ then
\begin{equation*}
 D_{\mu} \nu (x) = \lim_{r\downarrow 0} \frac{\nu (B_r(x))}{\mu(B_r(x))} \leq C \lim_{r\downarrow 0} \mu(B_r(x)) = 0. 
\end{equation*}
Letting $X:= \cup_j \{x_j\}$ we have  $D_{\mu} \nu = 0$ $\mu$-almost everywhere on $B \backslash X$.
Hence $D_{\mu} \nu$ is a simple function, therefore for Borel $E \subset B$
\begin{equation*}
 \nu(E) = \int_E D_{\mu} \nu \, d\mu = \sum_{\{j:x_j \in E\}} D_{\mu} \nu(x_j) \mu (\{x_j\}).
\end{equation*}
Setting $a_j := D_{\mu} \nu(x_j) \mu (\{x_j\})$ we have $\nu = \sum_j a_j \delta_{x_j}$. 
Now, for $\phi \in C_c^0(B)$
\begin{eqnarray*}
 \sum_j a_j \phi(x_j) &=& \lim_{n\rightarrow \infty} \int g_n^2 \phi \, dx \\
&=& \lim_{n\rightarrow \infty} \int (V_n - V)^2 \phi \, dx \\
&=& \lim_{n\rightarrow \infty} \left( \int (V_n^2 - V^2)\phi \, dx + 2\int V(V-V_n)\phi \, dx \right)
\end{eqnarray*}
where the last term vanishes in the limit since $V_n \rightharpoonup V$ weakly in $L^2$.
\end{proof}

\begin{lemma}[Corollary of Lemma \ref{measure}]\label{cmeasure}
 Suppose $\lbrace V_n \rbrace$ is as in Lemma \ref{measure}. If 
\begin{equation*}
 \lim_{r\downarrow 0} \limsup _{n\rightarrow \infty} \|V_n \|_{L^2(B_{r}(x))} =0
\end{equation*}
for all $x\in B$, then
\begin{equation*}
 V_n \rightarrow V 
\end{equation*}
strongly in $L^{2}_{loc}(B)$ (same $V$ as in Lemma \ref{measure}). 
\end{lemma}

\begin{proof}
First we apply Lemma \ref{measure} and viewing $|V_n|^2 dx$ as a sequence in $M(B)$ we notice that the condition $\lim_{r\downarrow 0} \limsup _{n\rightarrow \infty} \|V_n \|_{L^2(B_{r}(x))} =0 $ simply says that $V_n^2\rightharpoonup V^2$ weakly in $M(B)$.
%$a_j =0$ for all $j$. 
Therefore, given any open ball  $B_r(x)\subset\subset B$ we can apply standard results for Radon measures (\cite[\S 1.9 Theorem 1]{evans_gariepy}) to conclude that (since $\int_{\partial B_r(x)} |V|^2 dx =0$)  $\|V_n\|_{L^2(B_r(x))} \rightarrow \|V\|_{L^2(B_r(x))}$ for all $B_r(x)\subset\subset B$. Hence 
\begin{eqnarray*}
 \int_{B_r(x)} (V-V_n)^2 \, dx &=& \int_{B_r(x)} (V_n^2 - V^2)dx +2\int_{B_r(x)} V(V-V_n) \,dx \\
&\rightarrow& 0  \text{\, \, \, as \,}n\rightarrow \infty 
\end{eqnarray*}
since $V_n \rightharpoonup V$ weakly in $L^2$. 
Therefore $V_n \rightarrow V$ strongly in $L^2_{loc}(B)$.
\end{proof}

\subsection{Absorption lemma}

Special cases of the following lemma are widely used in 
regularity theory.

\begin{lemma} (Leon Simon 
\cite[\S 2.8, Lemma 2]{simon_regularity}.)
\label{LS}
Let $B_{\rho}(y) \subset \R^2$ be any ball, $k \in \R$, $\Gamma >0$, and let $\varphi$ be any $[0,\infty)$-valued convex subadditive function on the collection of convex subsets of $B_{\rho}(y)$; thus $\varphi(A) \leq \sum_{j=1}^{N} \varphi(A_j)$ whenever $A,A_1,A_2,....,A_N$ are convex subsets of $B_{\rho}(y)$ with $A\subset \bigcup_{j=1}^{N} A_j$. There is $\epsilon_0 = \epsilon_0 (k)$ such that if 
\begin{equation*}\label{leon}
\sigma^k \varphi(B_{\sigma /2}(z)) \leq \epsilon_0 \sigma^k \varphi(B_{\sigma}(z)) + \Gamma
\end{equation*}  
whenever $B_{2\sigma}(z) \subset B_{\rho}(y)$, then there exists some $C=C(k)<\infty$ such that
\begin{equation*}
\rho^k \varphi(B_{\rho /2}(y)) \leq C\Gamma.
\end{equation*} 
\end{lemma}

\bibliographystyle{plain}

\begin{thebibliography}{10}

\bibitem{adams_riesz}
David~R. Adams.
\newblock A note on {R}iesz potentials.
\newblock {\em Duke Math. J.}, 42(4):765--778, 1975.

\bibitem{clms}
R.~Coifman, P.-L. Lions, Y.~Meyer, and S.~Semmes.
\newblock Compensated compactness and {H}ardy spaces.
\newblock {\em J. Math. Pures Appl. (9)}, 72(3):247--286, 1993.

\bibitem{evans_weak_convergence}
Lawrence~C. Evans.
\newblock {\em Weak convergence methods for nonlinear partial differential
  equations}, {\em CBMS Regional Conference Series in
  Mathematics} {\bf 74}.
1990.

\bibitem{evans_gariepy}
Lawrence~C. Evans and Ronald~F. Gariepy.
\newblock {\em Measure theory and fine properties of functions}.
\newblock Studies in Advanced Mathematics. CRC Press, Boca Raton, FL, 1992.

\bibitem{feff_stein}
C.~Fefferman and E.~M. Stein.
\newblock {$H^{p}$} spaces of several variables.
\newblock {\em Acta Math.}, 129 137--193, 1972.

\bibitem{gt}
David Gilbarg and Neil~S. Trudinger.
\newblock {\em Elliptic partial differential equations of second order}.
\newblock Classics in Mathematics. Springer-Verlag, Berlin, 2001.

\bibitem{helein_regularity}
F. H{\'e}lein.
\newblock R\'egularit\'e des applications faiblement harmoniques entre une
  surface et une vari\'et\'e riemannienne.
\newblock {\em C. R. Acad. Sci. Paris S\'er. I Math.}, 312:591--596, 1991.

\bibitem{helein_conservation}
Fr{\'e}d{\'e}ric H{\'e}lein.
\newblock {\em Harmonic maps, conservation laws and moving frames}, 
{\em Cambridge Tracts in Mathematics} {\bf 150}.
\newblock Cambridge University Press, second edition, 2002.

\bibitem{hunt}
Richard~A. Hunt.
\newblock On {$L(p,\,q)$} spaces.
\newblock {\em Enseignement Math. (2)}, 12:249--276, 1966.

\bibitem{LZ}
J.~Li and X.~Zhu.
\newblock Small energy compactness for approximate harmonic mappings.
\newblock Preprint, 2009.

\bibitem{riviere_inventiones}
Tristan Rivi{\`e}re.
\newblock Conservation laws for conformally invariant variational problems.
\newblock {\em Invent. Math.}, 168(1):1--22, 2007.

\bibitem{riviere_notes}
Tristan Rivi\`ere.
\newblock Integrability by compensation in the analysis of conformally
  invariant problems.
\newblock Minicourse Notes, 2009.

\bibitem{rupflin_calcvar} \textsc{M. Rupflin},
\emph{An improved uniqueness result for the harmonic map 
flow in two dimensions.\/}
Calc. Var. {\bf 33} (2008) 329-Ð341.

\bibitem{semmes_primer}
Stephen Semmes.
\newblock A primer on {H}ardy spaces, and some remarks on a theorem of {E}vans
  and {M}\"uller.
\newblock {\em Comm. Partial Differential Equations}, 19(1-2):277--319, 1994.

\bibitem{simon_regularity}
Leon Simon.
\newblock {\em Theorems on regularity and singularity of energy minimizing
  maps}.
\newblock Lectures in Mathematics ETH Z\"urich. Birkh\"auser Verlag, Basel,
  1996.

\bibitem{stein_llogl}
E.~M. Stein.
\newblock Note on the class {$L$} {${\rm log}$} {$L$}.
\newblock {\em Studia Math.}, 32:305--310, 1969.

\bibitem{ziemer}
William~P. Ziemer.
\newblock {\em Weakly differentiable functions}, volume 120 of {\em Graduate
  Texts in Mathematics}.
\newblock Springer-Verlag, New York, 1989.

\end{thebibliography}

{\sc mathematics institute, university of Warwick, Coventry, CV4 7AL,
UK}\\

\end{document}